\documentclass[12pt]{amsart}
\usepackage[cp1250]{inputenc}
\usepackage[T1]{fontenc}
\usepackage{lmodern}

\usepackage{natbib}
\usepackage{amsmath,amsthm,amssymb}
\usepackage{pifont}
\usepackage{xfrac}
\usepackage[a4paper, left=2.2cm, right=2.2cm, top=2.3cm, bottom=2.3cm]{geometry}
\allowdisplaybreaks[1] 
\title{Existence of solution for a nonlinear model of thermo-visco-plasticity}

\author{Leszek~Bartczak}
\address{University of Warsaw, The Faculty of Mathematics, Informatics and Mechanics}
\email{L.Bartczak@mini.pw.edu.pl}
\thanks{The first autor was supported by the National Science Centre (decision number DEC-2011/03/N/ST1/04223).} 

\author{Sebastian~Owczarek}
\address{Warsaw University of~Technology, Faculty of~Mathematics and~Information Science}
\email{S.Owczarek@mini.pw.edu.pl}

\newtheorem{twr}{Theorem}
\newtheorem{lem}{Lemma}
\newtheorem{wni}{Corollary}
\newtheorem{stw}{Proposition}
\theoremstyle{definition}
\theoremstyle{definition}
\newtheorem{uwa}{Remark}

\newcommand{\pis}[1]{\mathcal{#1}}
\newcommand{\Rz}{\mathbb{R}}

\newcommand{\Sy}{\mathbb{S}(3)}
\newcommand{\dyw}{\mathrm{div}_x}
\newcommand{\nab}{\nabla_x}
\newcommand{\eps}{\varepsilon}
\newcommand{\epsp}{\varepsilon^{pl}}
\newcommand{\epsps}{\varepsilon^{pl,\star}}
\newcommand{\norm}[1]{\big\|#1\big\|}

\newcommand{\normsq}[1]{\left\|#1\right\|_{L^s(0,T;L^q(\Omega))}}
\newcommand{\dt}{\partial_t}
\newcommand{\dtt}{\partial^2_t}

\newcommand{\IO}{\int_\Omega}

\newcommand{\I}{\mathbb{I}}

\newcommand{\Ld}{\Lambda}

\newcommand{\ths}{\theta^{\star}}
\newcommand{\thsr}{{\theta}^{\star}_\Delta}

\newcommand{\Sr}{{\sigma}_\Delta}
\newcommand{\ur}{{u}_\Delta}
\newcommand{\epspr}{{\eps}^{pl}_\Delta}
\newcommand{\thr}{{\theta}_\Delta}

\newcommand{\Lp}{L^{p}(Q_T)}

\newcommand{\Lpr}{{L^p(0,T;L^r(\Omega))}}
\newcommand{\Lq}{L^{q}(\Omega)}
\newcommand{\Lqb}{L^{\bar{q}}(\Omega)}
\newcommand{\Wq}{W^{1,q}(\Omega)}
\newcommand{\Wqg}{W^{1,q}_{\Gamma_0}(\Omega)}
\newcommand{\Wqb}{W^{1,{\bar{q}}}(\Omega)}
\newcommand{\pn}{\frac{\partial\theta}{\partial \vec{n}}}
\begin{document}

\begin{abstract}We study a thermodynamically consistent model describing phenomena in a visco-plastic metal subjected to temperature changes. We complete the model with the mixed boundary condition on displacement and stress and Neumann-type condition for temperature. The main result is an existence of solution.
\end{abstract}
\date{}
\maketitle
\section{Introduction and main result}In this paper we propose an extension of the nonlinear thermo-visco-elasticity system (the mechanical dissipation is not linearised) to include a plastic effects. Our motivation for current considerations were the result of Blanchard and Guib{\'e} \cite{blangui}. In \cite{blangui} the authors considered the nonlinear thermo-visco-elasticity system and the existence of renormalised solutions for parabolic equations with $L^1$ data was proved (see also \cite{blan} and \cite{blanmur}). The case of small strains is analysed. Following Che{\l}mi\'{n}ski \cite{che}, Che{\l}mi{\'n}ski and Racke \cite{cherack}, Duvaut and J.L. Lions \cite{duvlions}, Gwiazda, Klawe and \'{S}wierczewska-Gwiazda \cite{gwklsw}, Suquet \cite{suquet1,suquet2}, Temam \cite{temam,temam1} and many others, we assume that displacement evolution is so slow that an acceleration term is negligible, so we only look at quasistatic form of the balance of momentum. In the considered model the inelastic constitutive equation is given by a Lipschitz-continuous function. The mechanical problem is coupled with the heat conduction equation. Changes of the temperature influence the stresses and the domain of plastic behaviour of the considered material. On the other hand stresses occurring in the body influence the heat production.
This type of problems are studied from many points of view (see {\it e.g.} \cite{RouBart1}, \cite{chehomb}, \cite{Paoli} or \cite{RouBart2}). See also the articles \cite{owcz1} and \cite{owcz2}, where a poroplasticity models are investigated (a poroplasticity models have a similar structure to the linear thermo-plasticity). In this paper we try tackle a thermomechanically consistent model without the linearisation in neighborhood of the reference temperature (compare \cite{bartczak}, \cite{bartczak1}, \cite{bartczak2}, \cite{cherack}).

We consider the body occupies initially a domain $\Omega\in\Rz^3$ with smooth boundary $\partial\Omega$. Let $x\in\Omega$ denote the material point while $t\in \Rz_+$ the time. Additionally we will denote by $\Sy$ the set of $3\times 3$ symmetric matrices of real entries. The system of equations is written in the following form 
\renewcommand{\arraystretch}{1.4}
\begin{equation*}\left\{\begin{array}{rcl}
-\dyw\sigma(t,x)&=&b(t,x), \\
\sigma(t,x)&=&\pis{D}(\eps(u(t,x))-\epsp(t,x))+\pis{C}(\eps(\dt u(t,x)))-\phi(\theta(t,x))\I,\\
\dt\epsp(t,x)&=&\Lambda(\sigma(t,x),\theta(t,x)),\\
\dt\theta(t,x)-\kappa\Delta\theta(t,x)&=&-\phi(\theta(t,x))\dyw\dt u(t,x)+\dt\epsp(t,x)\cdot\sigma(t,x)\\
&&+ \pis{C}(\eps(\dt u(t,x)))\cdot \eps(\dt u(t,x)).
\end{array}\right. \tag{TP}
\end{equation*}
The first equation describes the balance of momentum in the quasistatic case. The function $\sigma\colon\Rz_+\times\Omega\rightarrow\Sy$ is the stress tensor. The vector function $b\colon\Rz_+\times\Omega\rightarrow\Rz^3$ is a given density of volume forces. The second equation is the generalization of Hooke's law in Kelvin-Voight constrain. The function $u\colon\Rz_+\times\Omega\rightarrow\Rz^3$ describes displacement. We denote by  $\eps(u):=\frac{1}{2}\left(\nab u+(\nab u)^T\right)$ the symmetrical gradient of displacement also called the linear Cauchy strain tensor or simply small-strain tensor. It follows directly from definition that $\eps(u)\in \Sy$. Additionally we denote the plastic strain by $\epsp\colon\Rz_+\times\Omega\rightarrow \Sy$. The function $\theta\colon\Rz_+\times\Omega\rightarrow\Rz$ is the temperature. The function $\phi\colon\Rz\rightarrow\Rz$ defines so called thermal part of stress. We assume that it is a sublinear continuous function {\it i.e.} for all $s\in\Rz$
$$
|\phi(s)|\leq |s|^\alpha,
$$
where $\alpha>0$ will be indicated later.
By $\I$ we denote the identity $3\times 3$ matrix. We assume that operators $\pis{D}\colon\Sy\rightarrow\Sy$ and $\pis{C}\colon\Sy\rightarrow\Sy$ are given linear, symmetric and positive definite. In the next formula the evolution in time of plastic strain is given by an ordinary differential equation. Here the function $\Lambda:\Sy\times\Rz\rightarrow\Sy$ is Lipschitz continuous with a respect to the first argument and H\"older continuous with a respect to the second one {\it i.e.} there exist constants $L_1,L_2>0$ such that for all $A,B\in\Sy$ and $\varphi,\psi\in\Rz$ it holds that
$$
|\Lambda(A,\varphi)-\Lambda(B,\psi)|\leq L_1|A-B|+L_2|\varphi-\psi|^\beta,
$$
where $\beta>0$ will be indicated later.

The considered model includes non-linearities which are not globally Lipschitz. This is due to the fact that on the right-hand side of the heat conduction equation there are terms $\phi(\theta)\dyw\dt u$ and $ \dt \epsp \cdot \sigma $. They are the worst troublemakers due to its integrability and, in some models, if the acceleration term $\dtt u$ is neglected also term $ \dt \epsp \cdot \sigma $ is omitted on the right-hand side of the heat conduction equation.

We complete the system above with boundary conditions
\begin{equation*}\left\{\begin{array}{rll}
u(t,x)&=0,&\; \text{for}\; (t,x)\in(0,T)\times\Gamma_0,\\
\sigma(t,x)\cdot\vec{n}(x)&=g(t,x),&\; \text{for}\; (t,x)\in(0,T)\times\Gamma_1,\\
\pn(t,x)&=h(t,x),&\; \text{for}\; (t,x)\in(0,T)\times\partial \Omega
\end{array}\right.\tag{BC}
\end{equation*}
where $\Gamma_0,\Gamma_1\subset\partial\Omega$ satysfy $\Gamma_0\cap\Gamma_1=\emptyset$, $\Gamma_0\cup\Gamma_1=\partial\Omega$ and $\pis{H}^2(\Gamma_0)>0$ (here $\pis{H}^2$ denotes two-dimmensional Hausdorff measure). Let $\Gamma_0$ is relatively closed and $\Gamma_1$ is relatively open in $\partial\Omega$. Additionally, we will required that the set $\Omega\cup\Gamma_1$ is regular in sense of Gr\"oger ({\it cf.} \cite{groger}) {\it i.e.}  the set $\Omega\cup\Gamma_1$ is bounded and for every $x\in\partial\Omega$ there exist an open neighborhood $U\subset\Rz^3$ of the point $x$ in $\Rz^3$ and a bijective Lipschitz map $\omega\colon U\rightarrow\omega(U)\subset\Rz^3$ such that $\omega^{-1}$ is also Lipschitz continuous map and that the set $\omega\left(U\cap(\Omega\cup\Gamma_1)\right)$ takes a one of the following forms:
\begin{align*}
E_1&:=\left\{y\in\Rz^3\colon|y|<1,y_3<0\right\},\\
E_2&:=\left\{y\in\Rz^3\colon|y|<1,y_3\leq 0\right\},\\
E_3&:=\left\{y\in E_2\colon y_3<0\,\text{or}\,y_1>0\right\}.
\end{align*}
Finally we set initial conditions as: 
\begin{equation*}\left\{\begin{array}{rcl}
\epsp(0,x)&=\epsp_{0}(x),&\; \text{for}\; x\in\Omega,\\
u(0,x)&=u_0(x),&\; \text{for}\; x\in\Omega,\\
\theta(0,x)&=\theta_{0}(x),&\; \text{for}\; x\in\Omega.\end{array}\right.\tag{IC}\\[2ex]
\end{equation*}
Now let us introduce the function space which is very useful in considering of the mixed boundary problems as follows:
$$
\Wqg:=\overline{\left\{u|_\Omega\colon u\in C^\infty_0(\Rz^3)\,\text{and}\, \operatorname{supp} u\cap\Gamma_0=\emptyset\right\}},
$$
where the closeness is taken with respect to standard norm in the Sobolev space $\Wq$ \textit{i.e.} for $u\in\Wq$ we have
$$
\norm{u}_{\Wq}:=\left(\IO|u|^q\,dx+\sum_{i=1}^3\IO|\partial_{x_i}u|^q\,dx\right)^{\sfrac{1}{q}}.
$$
The mixed boundary condition for elasticity was studied by R. Herzog, C. Meyer and G. Wachsmuth. We recall the main fact which we will exploit in the further consideration {\it i.e.} Theorem 1.1 in \cite{herzog}:
\begin{twr}Let define the following operator $\mathbf{D}_s\colon W^{1,s}_{\Gamma_0}\rightarrow \left(W^{1,s'}_{\Gamma_0}(\Omega)\right)^\star$ as follows:
$$
\left[\mathbf{D}_s(u);v\right]:=\IO\pis{D}(\eps(u))\cdot\eps(v)\,dx\quad\text{for}\;u\in W^{1,s}_{\Gamma_0},v\in W^{1,s'}_{\Gamma_0}(\Omega).
$$
There exists $q>2$ such that for all $s\in[2,q]$ the operator $\mathbf{D}_s$ is continuously invertible. Moreover the inverse is globally Lipschitz with a Lipschitz constant independent of $s\in[2,q]$.\label{twr:herzog}
\end{twr}
\begin{uwa}It is worth to notice that it follows from the proof of the theorem mentioned above the constant $q$ depends only on a geometry of the domain $\Omega$ (also on the partition of the boundary) and entries of the operator $\pis{D}$. 

\noindent In the further investigation we will treat the constant $q$ ad fixed.
\end{uwa}
%
%
Now we formulate the main framework for further consideration. 
Let parameters $p,r,s,\alpha,\beta\in\Rz$ satisfy the following relations:
\begin{equation}
\left\{\begin{array}{rcl}
\infty&>&p,r>1,\\
\sfrac{1}{2}&>&\alpha,\beta>0,\\
r&\geq& \max\{\alpha,\beta\}\cdot q,\\
p&\geq&\max\{\alpha,\beta\}\cdot s.
\end{array}\right.
\label{eq:parametry}
\end{equation}
For $p,q,r,s,\alpha,\beta$ indicated above we introduce assumption on the given data:
\begin{itemize}
\item[{\bf A1}] Let $b\in L^s(0,T;L^q(\Omega;\Rz^3))$,
\item[{\bf A2}] Let $g\in W^{1,s}(0,T;W^{-1/q}(0,T)(\Gamma_1;\Rz^3))$,
\item[{\bf A3}] Let $h\in F^{(r-1)/(2r)}_{pr}(0,T;L^{q} (\partial\Omega)) \cap L^p(0,T;W^{1-1/r,r}(\partial\Omega))$,
\item[{\bf A4}] Let $\epsp_0\in W^{1,{q}}(\Omega;\Sy)$,
\item[{\bf A5}] Let $u_0\in W^{1,q}_{\Gamma_0}(\Omega;\Rz^3)$,
\item[{\bf A6}] Let $\theta_0\in B^{2(1-1/p)}_{rp}(\Omega)$.
\end{itemize}
Additionally in the case $r>3$ we have to assume a compatibility condition in the following form:
\begin{equation*}
\frac{\partial\theta_0(x)}{\partial \vec{n}}=h(0,x),\qquad\text{for}\;x\in\partial\Omega.\tag{CC}
\end{equation*}
For details see {\it e.g.} \cite{pruss}. 

The main theorem of current paper is
\begin{twr} Let the assumptions {\bf A1}-{\bf A6} are satisfy. Additionally if $r>3$ then we assume a compatibility condition (CC). Then there exists solution to the system (TP)+(BC)+(IC) $(u,\epsp,\theta)$ such that
\begin{align*}
u&\in W^{1,s}(0,T;\Wqg);\\
\epsp&\in W^{1,s}(0,T;\Lq);\\
\theta&\in L^r(0,T;W^{1,r}(\Omega))\cap W^{1,p}(0,T;L^r(\Omega).
\end{align*}
\label{twr:main}
\end{twr}
Theorem \ref{twr:main} prestent the existence result for the nonlinear thermo-visco-plsticity model with the mixed boundary conditions on displacement.
From the applications point of view the mixed boundary condition is very important. Moreover our motivation for the growth assumption on the function $\phi$ was taken from article \cite{blangui} in which the authors assume that the function $\phi$ satisfies $|\phi(r)|\leq C(1+|r|)^{\frac{1}{2}}$ for all $r\in\Rz_{-}$ and $C>0$. The main idea of the proof of Theorem \ref{twr:main} is a fixed point argument. 

%
%
%
%

%
%

%
%
\section{Existence of solution to TP}
In this section we are going to prove that there exists a solution $(u,\epsp,\theta)$ defined on $(0,T)\times\Omega$. We will handle with the problem of the existence using Schauder fixed point theorem. Before we start with the proof of Theorem \ref{twr:main}, we prove the following corollary comes from Theorem \ref{twr:herzog}.
\begin{wni}From Theorem \ref{twr:herzog} we can deduce a solvability to a problem of linear elasticity {\it i.e.} let $b\in W^{1,s}(0,T;(W^{1,q'}(\Omega;\Rz^3))^*)$ while $g\in W^{1,s}(0,T;W^{-1/q}(0,T)(\Gamma_1;\Rz^3))$. Moreover let $u_0\in\Wqg$. Then there exists a unique solution $u\in W^{1,s}(0,T;\Wqg)$ to the following problem
\begin{equation*}\left\{\begin{array}{rlrl}
-\dyw\pis{D}(\eps(u(t,x)))-\dyw\pis{C}(\eps(\dt u(t,x)))&=b(t,x),&\qquad \text{for}&\;(t,x)\in(0,T)\times\Omega,\\
u(t,x)&=0,&\; \text{in}&\; (t,x)\in(0,T)\times\Gamma_0,\\
\left[\pis{D}(\eps(u(t,x)))+\pis{C}(\eps(\dt u(t,x)))\right]\cdot\vec{n}(x)&=g(t,x),&\; \text{for}&\; (t,x)\in(0,T)\times\Gamma_1,\\
u(0,x)&=u_0(x),&\;\text{for}&\; x\in\Omega.
\end{array}\right. \tag{LE}
\end{equation*}
Furthermore the solution $u$ can be estimated as follows:
\begin{eqnarray}
\lefteqn{\norm{u}_{W^{1,s}(0,T;\Wqg}}\nonumber\\
&&\leq CT\left(e^{CT}+1\right)\left(\norm{u_0}_{\Wqg}+ \norm{b}_{W^{1,s}(0,T;(W^{1,q'}(\Omega))^*)}+ \norm{g}_{W^{1,s}(0,T;W^{-1/q}(\Gamma_1))}\right),
\label{eq:LEestim}
\end{eqnarray}
where a constant $C>0$ depends only on a geometry of the domain $\Omega$ (also on the partition of the boundary) and entries of operators $\pis{D}$ and $\pis{C}$. 
\label{wni:LE}
\end{wni}
\begin{proof}The proof follows the Banach fixed point theorem. We are going to construct a contractive operator $\pis{P}\colon L^{s}((0,T);\Wqg)\rightarrow L^{s}((0,T);\Wqg)$. Let $v\in L^{s}((0,T);\Wqg)$ and we look for the solution to the following problem:
\begin{equation*}\left\{\begin{array}{rlrl}
-\dyw\pis{C}(\eps(w))&=b+\dyw\pis{D}(\eps(v)),&\qquad \text{for}&\;(0,T)\times\Omega,\\
w&=0,&\; \text{on}&\; (0,T)\times\Gamma_0,\\
\pis{C}(\eps(w))\cdot\vec{n}&=g+\pis{D}(\eps(v))\cdot\vec{n},&\; \text{on}&\; (0,T)\times\Gamma_1,\\
w|_{t=0}&=u_1,&\;\text{in}&\; \Omega.
\end{array}\right. \tag{$\star$}
\end{equation*}
Obviously we can uniquely solve the problem ($\star$) as a straightforward conclusion from Theorem \ref{twr:herzog}. Thus we put $u(t,x)=\pis{P}(v)(t,x):=\int_0^t w(\tau,x)\,d\tau+u_0(x)$. Now we insert $v_1,v_2\in L^{s}((0,T);\Wqg)$ into system ($\star$) and obtain the solutions $w_1,w_2\in L^{s}((0,T);\Wqg)$. Therefore the difference $w_1-w_2$ satisfies the following system:
\begin{equation*}\left\{\begin{array}{rlrl}
-\dyw\pis{C}(\eps(w_1-w_2))&=\dyw\pis{D}(\eps(v_1-v_2)),&\qquad \text{for}&\;(0,T)\times\Omega,\\
w_1-w_2&=0,&\; \text{on}&\; (0,T)\times\Gamma_0,\\
\pis{C}(\eps(w_1-w_2))\cdot\vec{n}&=\pis{D}(\eps(v_1-v_2))\cdot\vec{n},&\; \text{on}&\; (0,T)\times\Gamma_1,\\
w_1-w_2|_{t=0}&=0,&\;\text{in}&\; \Omega.
\end{array}\right. \tag{$\star\star$}
\end{equation*}
Then using the Theorem \ref{twr:herzog} we conclude that for all $t\in(0,T)$ it holds that:
\begin{eqnarray}
\lefteqn{\norm{u_1-u_2}_{L^{s}((0,t);\Wqg)}}\nonumber\\
&&=\norm{\int_0^t (w_1(\tau,x)-w_2(\tau,x))\,d\tau}_{L^{s}((0,t);\Wqg)}\leq Ct\norm{v_1-v_2}_{L^{s}((0,t);\Wqg)}
\label{eq:1stbanach}
\end{eqnarray}
and we can claim that the operator $\pis{P}_1\colon L^{s}((0,T_1);\Wqg)\rightarrow L^{s}((0,T_1);\Wqg)$ is a contraction for $T_1=\frac{1}{2C}$. Hence by the Banach fixed point theorem we obtain that there exists an unique element $u\in L^{s}((0,T_1);\Wqg)$ such that $\pis{P}_1(u)=u$. Moreover from the construction of the operator $\pis{P}$ we immediately have that also $\dt u\in L^{s}((0,T_1);\Wqg)$. One can see that the estimate \ref{eq:1stbanach} does not depend on the initial condition thus we can repeat the reasoning above to obtain a sequence of contractive operators $\pis{P}_k\colon L^{s}((T_{k-1},T_k);\Wqg)\rightarrow L^{s}((T_{k-1},T_k);\Wqg)$, where $T_k=\frac{k}{2C}$ and corresponding solutions  $u\in W^{1,s}((T_{k-1},T_k);\Wqg)$.

It remains to prove the estimate (\ref{eq:LEestim}). First we integrate the system (LE) with respect to time to obtain:
\begin{equation*}\left\{\begin{array}{rcl}
-\dyw\pis{C}(\eps(u(t,x)))&=&-\dyw\pis{C}(\eps(u_0(x)))+\int_0^t \dyw\pis{D}(\eps(u(\tau,x)))\,d\tau+ \int_0^t b(\tau,x)\,d\tau,\\
u(t,x)|_{\Gamma_0}&=&0,\\
\pis{C}(\eps(u(t,x)))|_{\Gamma_1}\cdot\vec{n}(x)&=&\left(\pis{C}(\eps(u_0(x)))-\int_0^t \pis{D}(\eps(u(\tau,x)))\,d\tau\right)|_{\Gamma_1}\cdot\vec{n}(x)+\int_0^tg(\tau,x)\,d\tau,\\
u(0,x)&=&u_0(x).
\end{array}\right.
\end{equation*}
Therefore using Theorem \ref{twr:herzog} for any $t\in(0,T) $ we obtain the following estimate:
\begin{equation*}
\norm{u(t)}_{W^{1,q}_{\Gamma_0}(\Omega)}\leq C\Bigg(\norm{u_0}_{W^{1,q}_{\Gamma_0}(\Omega)}+ \int_0^t\norm{u(\tau)}_{W^{1,q}_{\Gamma_0}(\Omega)}d\tau+ \int_0^t\norm{b(\tau)}_{L^q(\Omega)}d\tau+ \int_0^t\norm{g(\tau)}_{W^{-1/q}(\Gamma_1)}d\tau\Bigg).
\end{equation*}
Hence using Gronwall inequality we obtain that:
\begin{equation*}
\norm{u(t)}_{W^{1,q}_{\Gamma_0}(\Omega)}\leq C\left(e^{Ct}+1\right)\Bigg(\norm{u_0}_{W^{1,q}_{\Gamma_0}(\Omega)}+ \int_0^t\norm{b(\tau)}_{L^q(\Omega)}+ \norm{g(\tau)}_{W^{-1/q}(\Gamma_1)}d\tau\Bigg)
\end{equation*}
and
\begin{equation*}
\norm{u}_{L^s(0,T);W^{1,q}_{\Gamma_0}(\Omega))}\leq CT\left(e^{CT}+1\right)\Bigg(\norm{u_0}_{W^{1,q}_{\Gamma_0}(\Omega)}+ \norm{b}_{L^s(0,T;L^q(\Omega))}+ \norm{g}_{L^s(0,T;W^{-1/q}(\Gamma_1))}\Bigg)
\end{equation*}
Now for any $t\in(0,T)$ we can estimate a norm of the deformation velocity as follows:
\begin{equation*}
\norm{\dt u(t)}_{W^{1,q}_{\Gamma_0}(\Omega)}\leq C\Bigg(\norm{u(t)}_{W^{1,q}_{\Gamma_0}(\Omega)}+\norm{b(t)}_{L^q(\Omega)}+ \norm{g(t)}_{W^{-1/q}(\Gamma_1)}\Bigg),
\end{equation*}
therefore
\begin{equation*}
\norm{\dt u}_{L^s(0,T;W^{1,q}_{\Gamma_0}(\Omega))}\leq CT\left(e^{CT}+1\right)\Bigg(\norm{u_0}_{W^{1,q}_{\Gamma_0}(\Omega)}+ \norm{b}_{L^s(0,T;L^q(\Omega))}+ \norm{g}_{L^s(0,T;W^{-1/q}(\Gamma_1))}\Bigg)
\end{equation*}
and the proof is completed.
\end{proof}
%
%
\subsection{Fixed point} We shall construct the compact operator $\pis{T}\colon \Lpr\rightarrow\Lpr $. Let us fix a function $\ths\in \Lpr$ and consider the first auxiliary problem:
\begin{equation*}\left\{\begin{array}{rcl}
-\dyw\sigma(t,x)&=&b(t,x), \\
\sigma(t,x)&=&\pis{D}(\eps(u(t,x))-\epsp(t,x))+\pis{C}(\eps(\dt u(t,x)))-\phi(\ths(t,x))\I,\\
\dt\epsp(t,x)&=&\Ld(\sigma(t,x),\ths(t,x)), \\
u(t,x)|_{\Gamma_0\times(0,T)}&=&0,\\
\sigma(t,x)\cdot\vec{n}(x)|_{\Gamma_1\times(0,T)}&=&g(t,x),\\
\epsp(0,x)&=&\epsp_0(x),\\
u(0,x)&=&u_0(x).\end{array}\right.\tag{AP1}
\end{equation*}
\begin{lem}Assume that {\bf A1}-{\bf A2} and {\bf A5}-{\bf A6} are satisfied and moreover $\ths\in \Lpr$. Then there exists the unique solution $(\sigma, \epsp, u)$ to (AP1) satysfying $\sigma\in L^s(0,T;L^ {q}(\Omega;\Sy))$ and  $\epsp\in W^{1,s}(0,T;L^ {q}(\Omega;\Sy))$ while $u\in W^{1,s}(0,T;W^{1, {q}}(\Omega;\Rz^3))$. Additionally the following estimate holds
\begin{eqnarray}
\lefteqn{\norm{\sigma}_{L^\infty(0,T;L^ {q}(\Omega))}+\norm{\epsp}_{W^{1,s}(0,T;L^ {q}(\Omega))}+ \norm{u}_{W^{1,s}(0,T;\Wqg)}}\nonumber\\
&\quad&\leq E(T)\left(\norm{\ths}^{\beta}_{L^{\beta s}(0,T;L^{\beta q}(\Omega))}+ \norm{\ths}^{\alpha}_{L^{\alpha s}(0,T;L^{\alpha q}(\Omega))}+|\Omega||\Ld(0,0)|+\norm{\epsp_0}_{\Lq}\right.\label{eq:ap1a}\\
&&\phantom{E(T)}\left.\norm{u_0}_{\Wqg}+ \norm{b}_{W^{1,s}(0,T;\Lqb)}+ \norm{g}_{W^{1,s}(0,T;W^{-\frac{1}{ {q}}, {q}}(\partial\Omega))}\right).\nonumber
\end{eqnarray}
\label{lem:ap1}
\end{lem}
\begin{proof} We will again use the Banach fixed point theorem. Let $\epsps\in L^s(0,T;L^q(\Omega;\Sy))$ and we solve the linear elasticity problem in the form:
\begin{equation}
\left\{\begin{array}{rlrl}    
-\dyw\pis{D}(\eps(w))-\dyw\pis{C}(\eps(\dt w))&= \dyw\pis{D}(\epsps)+\nab\phi(\ths)+b,\quad&\text{on}\;& (0,T)\times \Omega,\\
w&=0, &\text{on}\;&(0,T)\times\Gamma_0,\\
\left[\pis{D}(\eps(w))+\pis{C}(\eps(\dt w))\right]\cdot\vec{n}&= \left[\pis{D}(\epsps)+\I\phi(\ths)\right]\cdot\vec{n}+ g,&\text{on}\;&(0,T)\times\Gamma_1.
\end{array}\right.
\label{eq:elast}
\end{equation}
By Corollary \ref{wni:LE} we obtain that there exists a unique solution $w\in W^{1,s}(0,T;W^{1,q}(\Omega;\Rz^3))$ satisfying the following estimate
\begin{eqnarray*}
\norm{w}_{W^{1,s}(0,T;\Wqg)}&\leq& CT\left(e^{CT}+1\right)\Big(\normsq{\epsps}+ \normsq{\phi(\ths)}\\
&&\phantom{ C\Big(}+\norm{u_0}_{\Wqg}+ \normsq{b}+ \norm{g}_{L^s(0,T;W^{-\frac{1}{q},q}(\partial\Omega))}\Big),
\end{eqnarray*}
where the constant $C>0$ depends on entries of operators $\pis{D}$, $\pis{C}$ and the geometry of the set $\Omega$. Let us denote $D(T):=CT\left(e^{CT}+1\right)$ and let us put $$\sigma=\pis{D}(\eps(w)-\epsps)+\pis{C}(\eps(\dt w))-\phi(\ths).$$ One sees that $\dyw\sigma=b$, thus obviously
\begin{eqnarray*}
\lefteqn{\normsq{\sigma}+\normsq{\dyw\sigma}}\\
&\leq& C\left(\norm{w}_{W^{1,s}(0,T; \Wqg )}+ \normsq{\epsps}\right)+\normsq{b}\\
&\leq& D(T)\Big(\normsq{\epsps}+ \normsq{\phi(\ths)}+\norm{u_0}_{\Wqg}+ \normsq{b}\\
&&\phantom{ D(T)\Big( }+ \norm{g}_{L^s(0,T;W^{-\frac{1}{q},q}(\partial\Omega))}\Big).
\end{eqnarray*}
We define $\epsp(t,x):=\int_0^t\Lambda(\sigma(x,\tau),\ths(x,\tau))d\tau+ \epsp_0(x)$, so we obtain the following estimate
\begin{eqnarray*}
\lefteqn{}\\
|\epsp(t,x)|&\leq&\int_0^t\left|\Ld(\sigma(x,\tau),\ths(x,\tau))\right|d\tau  +|\eps^p_0(x)|\\
&\leq&\int_0^t\left|\Ld(\sigma(x,\tau),\ths(x,\tau))-\Ld(0,0)\right|d\tau + t|\Ld(0,0)|+ |\eps^p_0(x)|\\
&\leq&\int_0^tL_1\left|\sigma(x,\tau)\right|+L_2\left|\ths(x,\tau))\right|^{\beta}d\tau + t|\Ld(0,0)|+  
|\eps^p_0(x)|.
\end{eqnarray*}
Hence we can estimate as follows
$$\normsq{\epsp}\leq CT\left(L_1\normsq{\sigma}+|\Omega||\Ld(0,0)|\right)+ L_2\norm{\ths}^{\beta}_{L^{\beta s}0,T;L^{\beta q}(\Omega)}+ \norm{\eps^p_0}_{\Lq}.
$$
Now we define the operator $\pis{R}:L^s(0,T;L^q(\Omega;\Sy))\rightarrow L^s(0,T;L^q(\Omega;\Sy))$ as $\pis{R}(\epsps):=\epsp$. We claim that we can choose such a short time interval that the operator $\pis{R}$ is a contraction. Indeed, suppose that $\epsps_1,\epsps_2\in L^s(0,T;L^q(\Omega;\Sy))$ while $w_1, w_2$ respectively are solutions to \ref{eq:elast} then we can obtain similarly as previously
\begin{eqnarray*}
\epsp_1(t,x)&=&\int_0^t\Ld(\sigma_1(x,\tau),\ths(x,\tau))d\tau +\eps^p_0(x),\\
\epsp_2(t,x)&=&\int_0^t\Ld(\sigma_2(x,\tau),\ths(x,\tau))d\tau +\eps^p_0(x),
\end{eqnarray*}
where $\sigma_i=\pis{D}(\eps(w_i)-\epsps_i)+\pis{C}(\eps(\dt w))-\phi(\ths)$ for $i=1,2$.
We subtract $\epsp_1-\epsp_2$ to obtain
$$\norm{\epsp_1(t)-\epsp_2(t)}_{\Lq}\leq\int_0^t\norm{\Ld(\sigma_1(\tau),\ths(\tau))-\Ld(\sigma_2(\tau),\ths(\tau))}_{\Lq}d\tau\leq L_1\int_0^t\norm{\sigma_1(\tau)-\sigma_2(\tau)}_{\Lq}d\tau
$$
hence
$$
\normsq{\epsp_1-\epsp_2}\leq L_1T\normsq{\sigma_1-\sigma_2}.
$$
Since $w_1$ and $w_2$ are solutions to \eqref{eq:elast} with terms $\epsps_1$ and $\epsps_2$ on the right hand side respectively. It follows from the linearity of the problem \eqref{eq:elast} that the difference $w_1-w_2$ satisfies
\begin{equation*}
\left\{\begin{array}{rlrl}
-\dyw\pis{D}(\eps(w_1-w_2))-\dyw\pis{C}(\eps(\dt (w_1- w_2)))&=-\dyw\pis{D}(\epsp_1-\epsp_2),&\text{in}\;& (0,T)\times\Omega,\\
w_1-w_2&=0,&\text{on}\;& (0,T)\times\Gamma_0,\\
\left[\pis{D}(\eps(w_1-w_2))+\pis{C}(\eps(\dt (w_1-w_2)))\right]\cdot\vec{n}&= \pis{D}(\epsps_1-\epsps_2)\cdot\vec{n},\quad&\text{on}\;& (0,T)\times\Gamma_1,\\
w_1-w_2&=0,&\text{on}\;& \Omega\times\{t=0\}\\
\end{array}\right.
\end{equation*}
and the following inequality holds
$$\norm{w_1-w_2}_{W^{1,s}(0,T; \Wq )}\leq D(T)\normsq{\epsps_1-\epsps_2},
$$
where the constant $D(T)$ is the same as previously. Hence
\begin{eqnarray*}
\lefteqn{\normsq{\sigma_1-\sigma_2}+\normsq{\dyw(\sigma_1-\sigma_2)}}\\
&&\leq C\left(\norm{w_1-w_2}_{W^{1,s}(0,T; \Wqg )}+ \normsq{\epsps_1-\epsps_2}\right)\\
&&\leq CD(T)\normsq{\epsps_1-\epsps_2}.
\end{eqnarray*}
Considering the difference $\pis{R}(\epsps_1)-\pis{R}(\epsps_2)$ we can obtain
\begin{equation}
\normsq{\pis{R}(\epsps_1)-\pis{R}(\epsps_2)}= \normsq{\epsp_1-\epsp_2} \leq CTD(T)\normsq{\epsps_1-\epsps_2},
\label{eq:contr1}
\end{equation}
where the constant $C>0$ does not depend on time $T$ and the initial data while $D(T)$ is the same as previously and also does not depend on initial data. Choosing properly small $T_1>0$ we get that $\pis{R}_1$ is a contraction on $L^s(0,T_1;L^q(\Omega;\Sy))$. Therefore by the Banach fixed point theorem we obtain that there exists $\epsp\in L^s(0,T_1;L^q(\Omega;\Sy))$ such that $\pis{R}_1(\epsp)=\epsp$ and the problem (AP1) possesses the unique solution $(\sigma, \epsp,u)\in L^s(0,T_1;L^q(\Omega;\Sy))\times L^s(0,T_1;L^q(\Omega;\Sy))\times W^{1,s}(0,T;\Wqg)$ on the time interval $(0,T_1)$. Using the reasoning analogous to the proof of Corollary \ref{wni:LE} we can extend our solution to the whole interval $(0,T)$ since the estimate \eqref{eq:contr1} is independent of the initial data.

It remains to prove the estimate \eqref{eq:ap1a}. Thus using the second equation in (AP1) we obtain the following inequalities
\begin{eqnarray}
\lefteqn{\norm{\sigma(t)}_{\Lq}}\nonumber\\
& \leq & C\left(\norm{u(t)}_{\Wqg}+\norm{\dt u(t)}_{\Wqg}+ \norm{\epsp(t)}_{\Lq}+\norm{\phi(\ths(t))}_{\Lq}\right)\nonumber\\
& \leq & \widetilde{D}(T)\left(\norm{\epsp(t)}_{\Lq}+ \norm{\phi(\ths(t))}_{\Lq} + \norm{u_0}_{\Wqg}\right.\nonumber\\
&&\left.\phantom{D(T)\Big(}+ \norm{\epsp_0}_{\Lq}+ \norm{b(t)}_{\Lq}+ \norm{g(t)}_{W^{-\frac{1}{q},q}(\partial\Omega)}\right).\nonumber\\
& \leq & \widetilde{D}(T)\left(L_1\int_0^t\norm{\sigma(\tau)}_{\Lq}+ L_2\norm{\ths(\tau)}^{\beta}_{L^{\beta q}(\Omega)}\,d\tau+ \norm{\ths(t)}^{\alpha}_{L^{\alpha q}(\Omega)} + t|\Omega||\Ld(0,0)|\right.\\
&&\left.\phantom{D(T)\Big(}+ \norm{u_0}_{\Wqg}+ \norm{\epsp_0}_{\Lq}+ \norm{b(t)}_{\Lq}+ \norm{g(t)}_{W^{-\frac{1}{q},q}(\partial\Omega)}\right).\nonumber
\end{eqnarray}
The Gronwall inequality applied to the expression above implies
\begin{eqnarray}
\norm{\sigma(t)}_{\Lq} &\leq& E_1(T)\left(\norm{\ths}^{\beta}_{L^{\beta s}(0,T;L^{\beta q}(\Omega))}+ \norm{\ths}^{\alpha}_{L^{\alpha s}(0,T;L^{\alpha q}(\Omega))} + |\Omega||\Ld(0,0)|+\norm{\epsp_0}_{\Lq}\right.\label{eq:sig1}\\
&&\phantom{E_1(T)\Big(}\left.+ \norm{u_0}_{\Wqg}+ \normsq{b}+  \norm{g}_{L^s(0,T;W^{-\frac{1}{q},q}(\partial\Omega))}\right)\nonumber
\end{eqnarray}
and immediately we obtain the required estimate for $\normsq{\sigma}$. To estimate term $\normsq{\dt\epsp}$ we treat as follows
\begin{eqnarray}
\lefteqn{\normsq{\dt\epsp}}\nonumber\\
&\leq&E_2(T)\left(L_1\normsq{\sigma}+|\Omega||\Ld(0,0)|\right)+ L_2\norm{\ths}^{\beta}_{L^{\beta s}(0,T;L^{\beta q}(\Omega))}+ \norm{\eps^p_0}_{\Lq}\nonumber\\
&\leq&E_3(T)\left(\norm{\ths}^{\beta}_{L^{\beta s}(0,T;L^{\beta q}(\Omega))}+ \norm{\ths}^{\alpha}_{L^{\alpha s}(0,T;L^{\alpha q}(\Omega))} + |\Omega||\Ld(0,0)|+\norm{\epsp_0}_{\Lq}\right.\label{eq:epsp1}\\
&&\phantom{E_3(T)\Big(}\left.+\norm{u_0}_{\Wqg}+ \normsq{b}+ \norm{g}_{L^s(0,T;W^{-\frac{1}{q},q}(\partial\Omega))}\right)\nonumber.
\end{eqnarray}
To obtain the estimate for $\norm{u}_{W^{1,s}(0,T;\Wqg}$ we once again use the Corollary \ref{wni:LE} as in problem \eqref{eq:elast}. Therefore
\begin{eqnarray}
\lefteqn{\norm{u}_{W^{1,s}(0,T;\Wqg)}}\nonumber\\
&\leq& D(T)\Bigg(\norm{\dt\epsp}_{L^s(0,T;L^q(\Omega))}+  \norm{\ths}^{\alpha}_{L^{\alpha s}(0,T;L^{\alpha q}(\Omega))}\nonumber\\
&&\phantom{D(T)\Bigg(}+ \normsq{b}+ \norm{g}_{W^{1,s}\infty(0,T;W^{-\frac{1}{q},q}(\partial\Omega))}+ \norm{u_0}_{\Wqg}\Bigg)\nonumber\\
&\leq& E(T)\Bigg(\norm{\ths}^{\beta}_{L^{\beta s}(0,T;L^{\beta q}(\Omega))}+ \norm{\ths}^{\alpha}_{L^{\alpha s}(0,T;L^{\alpha q}(\Omega))}+ |\Omega||\Ld(0,0)|+\norm{\epsp_0}_{\Lq}\label{eq:dtu}\\
&&\phantom{E(T)\Bigg(}+ \normsq{b}+  \norm{g}_{W^{1,s}(0,T;W^{-\frac{1}{q},q}(\partial\Omega))}+ \norm{u_0}_{\Wqg}\Bigg)\nonumber
\end{eqnarray}
The inequalities \eqref{eq:sig1} together with \eqref{eq:epsp1} and \eqref{eq:dtu} give us \eqref{eq:ap1a}.
\end{proof}
%
%
Next for $\ths\in L^p(0,T;L^r(\Omega))$ and $(\sigma,\epsp,u)$ the corresponding solution to (AP1) we solve the second auxiliary problem formulated as follows
\begin{equation*}\left\{\begin{array}{rcl}
\dt\theta(t,x)-\kappa\Delta\theta(t,x)&=&-\phi(\ths(t,x))\dyw\dt u(t,x)+\dt\epsp(t,x)\cdot\sigma(t,x)\\
&&+ \pis{C}(\eps(\dt u(t,x)))\cdot \eps(\dt u(t,x)),\\
\pn(t,x)|_{\partial\Omega\times(0,T)}&=&h(t,x),\\
\theta(x,0)&=&\theta_0(x).\end{array}\right.\tag{AP2}
\end{equation*}
The problem above is the linear heat conduction equation and its solvability is proved for example in \cite{pruss}. We only need to estimate a norm of the following therm:
$$
RHS:=-\phi(\ths)\dyw\dt u+\dt\epsp\cdot\sigma+ \pis{C}(\eps(\dt u))\cdot \eps(\dt u)
$$
in the space $L^p(0,T;L^r(\Omega))$.
\begin{lem}For $RHS$ defined as above it holds that
\begin{equation*}
\norm{RHS}_{L^p(0,T;L^r(\Omega))}\leq D\left(\norm{\ths}^{2\beta}_{L^{p}(0,T;L^{r}(\Omega))}+ \norm{\ths}^{2\alpha}_{L^{p}(0,T;L^{r}(\Omega))}+ \norm{\ths}^{\beta+\alpha}_{L^{p}(0,T;L^{r}(\Omega))}+1\right),
\end{equation*}
where constant $D>0$ depends only on a length of a time interval $[0,T]$ and given data.
\label{lem:RHS}
\end{lem}
\begin{proof}Using estimates that was obtained before we can calculate as follows:
\begin{eqnarray*}
\lefteqn{\norm{RHS}_{L^p(0,T;L^r(\Omega))}}\\
&\leq&\norm{\phi(\ths)\dyw\dt u}_{L^p(0,T;L^r(\Omega))}+ \norm{\dt\epsp\cdot\sigma}_{L^p(0,T;L^r(\Omega))}+ \norm{\pis{C}(\eps(\dt u))\cdot \eps(\dt u)}_{L^p(0,T;L^r(\Omega))}\\
&\leq& \norm{\phi(\ths)}_{L^{2p}(0,T;L^{2r}(\Omega))}\norm{\dyw\dt u}_{L^{2p}(0,T;L^{2r}(\Omega))}+ \norm{\dt\epsp}_{L^{2p}(0,T;L^{2r}(\Omega))}\norm{\sigma}_{L^{2p}(0,T;L^{2r}(\Omega))}\\
&&+C\norm{\nab\dt u}^2_{L^{2p}(0,T;L^{2r}(\Omega))}\\
&\leq& C\Bigg(\norm{\ths}^\beta _{L^{2\beta p}(0,T;L^{2\beta r}(\Omega))}\norm{u}_{W^{1,s}(0,T;\Wqb)}+ \normsq{\dt\epsp}\normsq{\sigma}\\
&&\phantom{C\Big(}+ \norm{u}^2_{W^{1,s}(0,T;W^{1,q}(\Omega))}\Bigg)\\
&\leq& C(T) \norm{\ths}^\beta _{L^{p}(0,T;L^{r}(\Omega))}\Bigg(\norm{\ths}^{\beta}_{L^{\beta s}(0,T;L^{\beta q}(\Omega))}+ \norm{\ths}^{\alpha}_{L^{\alpha s}(0,T;L^{\alpha q}(\Omega))}+|\Omega||\Ld(0,0)|+\norm{\epsp_0}_{\Lq}\\
&&\phantom{C(T) \norm{\ths}^\beta _{L^{p}(0,T;L^{r}(\Omega))}\Bigg(}+\norm{b}_{W^{1,s}(0,T;\Lq)}+ \norm{g}_{W^{1,s}(0,T;W^{-\frac{1}{q},q}(\partial\Omega))}\Bigg)\\
&&+ C(T) \Bigg(\norm{\ths}^{\beta}_{L^{\beta s}(0,T;L^{\beta q}(\Omega))}+ \norm{\ths}^{\alpha}_{L^{\alpha s}(0,T;L^{\alpha q}(\Omega))}+ |\Omega||\Ld(0,0)|+\norm{\epsp_0}_{\Lq}\\
&&\phantom{+ C(T) \Bigg(}+\norm{u_0}_{\Wqg}+ \norm{b}_{W^{1,s}(0,T;\Lq)}+ \norm{g}_{W^{1,s}(0,T;W^{-\frac{1}{q},q}(\partial\Omega))}\Bigg)^2\\
&\leq& C(T) \Bigg(\norm{\ths}^{2\beta}_{L^{p}(0,T;L^{r}(\Omega))}+ \norm{\ths}^{2\alpha}_{L^{p}(0,T;L^{r}(\Omega))}+ \norm{\ths}^{\beta+\alpha}_{L^{p}(0,T;L^{r}(\Omega))} +\norm{u_0}^2_{\Wqg} + |\Omega|^2|\Ld(0,0)|^2\\
&&\phantom{C(T) \Bigg(}+\norm{\epsp_0}^2_{\Lq}+ \norm{b}^2_{W^{1,s}(0,T;\Lq)}+ \norm{g}^2_{W^{1,s}(0,T;W^{-\frac{1}{q},q}(\partial\Omega))}\Bigg).%
\end{eqnarray*}

\end{proof}
%
%
\begin{lem} Assume that {\bf A3} and {\bf A6} are satisfied and $\ths\in L^p(0,T;L^r(\Omega))$. Additionally if $r>3$ we assume the compatibility condition $CC$. Moreover let $\sigma\in L^s(0,T;L^q(\Omega;\Sy))$, $\epsp\in W^{1,s}(0,T;L^q(\Omega;\Sy))$ and $u\in W^{1,s}(0,T;W^{1,q}_{\Gamma_0}(\Omega;\Rz^3))$ be the solution to (AP1). Then the problem (AP2) possesses a unique solution $\theta\in L^p(0,T;W^{2,r}(\Omega))\cup W^{1,p}(0,T;L^r(\Omega))$. Furthermore the following estimate holds:
\begin{eqnarray}
\lefteqn{\norm{\theta}_{L^p(0,T;W^{2,r}(\Omega))}+\norm{\theta}_{W^{1,p}(0,T;L^r(\Omega))}}{}\nonumber\\
&\leq& C(T)\Big(\norm{RHS}_{L^p(0,T;L^r(\Omega))}+ \norm{\theta_0}_{B^{2(1-1/p)}_{rp}(\Omega)}+ \norm{h}_{F^{(r-1)/(2r)}_{pr}(0,T;L^r(\partial\Omega))}\label{eq:ap2}\\
&&\phantom{C(T)\Big(} + \norm{h}_{L^p(0,T;W^{1-1/r,r}(\partial\Omega))}\Big).\nonumber
\end{eqnarray}
\label{lem:ap2}
\end{lem}
%
%
By solving problems (AP1) and (AP2) we can define the operator 
\begin{align*}
\pis{T}\colon L^p(0,T;L^r(\Omega))&\longrightarrow  L^p(0,T;W^{2,r}(\Omega))\cup W^{1,p}(0,T;L^r(\Omega));\\
\pis{T}\colon \ths&\longmapsto\theta.
\end{align*}
\begin{stw}
The operator $\pis{T}\colon \Lp\rightarrow  L^p(0,T;W^{2,r}(\Omega))\cup W^{1,p}(0,T;L^r(\Omega))$ defined above is continuous.
\label{stw:T-ciagly}
\end{stw}
\begin{proof}To prove continuity of the operator $\pis{T}$ we choose $\ths_1,\ths_2\in\Lq$ and we study the difference $\pis{T}(\ths_1)-\pis{T}(\ths_2)$. First we consider $(\sigma_1,\epsp_1,u_1)$ and $(\sigma_2,\epsp_2,u_2)$ solutions to (AP1) corresponding to $\ths_1,\ths_2$ with the same given data $b,f,g$. Let denote $\thsr:=\ths_1-\ths_2$ and, respectively, $\Sr:=\sigma_1-\sigma_2$, $\epspr:=\epsp_1-\epsp_2$ and $\ur:=u_1-u_2$. Notice that the functions $\Sr, \epspr,\ur$ satisfy the following system of the equations:
\begin{equation*}\left\{\begin{array}{rcl}
-\dyw\Sr(t,x)&=&0, \\
\Sr(t,x)&=&\pis{D}(\eps(\ur(t,x))-\epspr(t,x))+ \pis{C}(\eps(\dt \ur(t,x)))\\
&&-(\phi(\ths_1(t,x))-\phi(\ths_2(t,x)))\I,\\
\dt\epspr(t,x)&=&\Ld(\sigma_1,\ths_1)-\Ld(\sigma_2 ,\ths_2), \\
\ur(t,x)|_{\Gamma_0\times(0,T)}&=&0,\\
\Sr(t,x)\cdot\vec{n}(x)|_{\Gamma_1\times(0,T)}&=&0,\\
\ur(0,x)&=&0,\\
\epspr(x,0)&=&0.\end{array}\right.
\end{equation*}
Hence we can estimate as follows:
$$
|\epspr(t,x)|\leq L_1\int_0^t|\Sr(x,\tau)|d\tau+ L_2\int_0^t|\thsr(t,\tau)|^\beta d\tau.
$$
Additionally $\ur$ satisfies:
\begin{equation*}
\left\{\begin{array}{rcl}
\dyw\pis{D}(\eps(\ur(t,x)))+\dyw+\pis{C}(\eps(\dt \ur(t,x)))&=&\dyw\pis{D}(\epspr(t,x)),\\
\ur(t,x)|_{\Gamma_0\times(0,T)}&=&0,\\
\big(\pis{D}(\eps(\ur(t,x)))+\pis{C}(\eps(\dt \ur(t,x)))\big)\cdot\vec{n}(x)|_{\Gamma_1\times(0,T)}&=& \big(\pis{D}(\epspr(t,x))\\
&&+ (\phi(\ths_1(t,x))-\phi(\ths_2(t,x)))\I\big)\cdot\vec{n}(x)|_{\Gamma_1\times(0,T)},
\end{array}\right.
\end{equation*}
and similarly as in the proof of Lemma \ref{lem:ap1} we can estimate
\begin{eqnarray}
\norm{\ur}_{W^{1,s}(0,T;\Wqg}&\leq& C(T)\left(\norm{\epspr}_{W^{1,s}(0,T;L^q(\Omega))}+ \norm{\phi(\ths_1)-\phi(\ths_2)}_{L^s(0,T;\Lq)}\right)\nonumber\\
&\leq&C(T)\Big(\norm{\Sr}_{L^{s}(0,T;\Lq)}+\norm{\thsr}^\beta_{L^{\beta s}(0,T;L^{\beta q}(\Omega))}\label{eq:u_roz}\\
&&\phantom{C(T)\Big(}+ \norm{\phi(\ths_1)-\phi(\ths_2)}_{L^s(0,T;\Lq)}\Big).\nonumber
\end{eqnarray}
Next, a similar calculations as in (\ref{eq:sig1}) and (\ref{eq:epsp1}) lead us to:
\begin{equation*}
\norm{\Sr}_{L^s(0,T;\Lq}+ \norm{\epspr}_{L^s(0,T;\Lq)}\leq C(T)\left(\norm{\thsr}^\beta_{L^{\beta s}(0,T;L^{\beta q}(\Omega))}+ \norm{\phi(\ths_1)-\phi(\ths_2)}_{L^s(0,T;\Lq)}\right)\label{eq:S_epsp_roz}.
\end{equation*}
The estimate (\ref{eq:u_roz}) together with (\ref{eq:S_epsp_roz}) give us:
\begin{eqnarray}
\lefteqn{\norm{\ur}_{W^{1,s}(0,T;\Wqg}+ \norm{\Sr}_{L^s(0,T;\Lq}+ \norm{\epspr}_{L^s(0,T;\Lq)}}\nonumber\\
&&\leq C(T)\left(\norm{\thsr}^\beta_{L^{\beta s}(0,T;L^{\beta q}(\Omega))}+ \norm{\phi(\ths_1)-\phi(\ths_2)}_{L^s(0,T;\Lq)}\right).
\end{eqnarray}
Moreover one can observe that operator defined by function $\phi$ as $L^p(0,T;L^r(\Omega))\ni \theta\mapsto\phi(\phi)\in L^s(0,T;\Lq)$ is bounded and continuous. Therefore for any $\theta_1\in L^p(0,T;L^r(\Omega)$ and any $\epsilon>0$ there exists $\delta>0$ such that if $\norm{\thsr}_{L^p(0,T;L^r(\Omega)}<\delta$ then it holds that:
\begin{equation}
\norm{\ur}_{W^{1,s}(0,T;\Wqg}+ \norm{\Sr}_{L^s(0,T;\Lq}+ \norm{\epspr}_{L^s(0,T;\Lq)}\leq \epsilon.\label{eq:ciaglosc1}
\end{equation}
Next we consider the heat conduction equation (in fact the problem similar to (AP2)) for $RHS$ defined by $u_i,\sigma_i,\epsp_i,\ths_i$ where $i=1,2$:
\begin{equation*}\left\{\begin{array}{rcl}
\dt\theta_i(t,x)-\kappa\Delta\theta_i(t,x)&=&-\phi(\ths_i(t,x))\dyw\dt u_i(t,x)+\dt\epsp_i(t,x)\cdot\sigma_i(t,x)\\
&&+ \pis{C}(\eps(\dt u_i(t,x)))\cdot \eps(\dt u_i(t,x)),\\
\pn(t,x)|_{\partial\Omega\times(0,T)}&=&h(t,x),\\
\theta(x,0)&=&\theta_0(x)\end{array}\right.\qquad i=1,2.
\end{equation*}
Obviously the difference $\thr$ solves the following problem:
\begin{equation*}\left\{\begin{array}{rcl}
\dt\thr(t,x)-\kappa\Delta\thr(t,x)&=&-\phi(\ths_1(t,x))\dyw\dt u_1(t,x)+ \phi(\ths_2(t,x))\dyw\dt u_2(t,x)\\ &&+\dt\epsp_1(t,x)\cdot\sigma_1(t,x)- \dt\epsp_2(t,x)\cdot\sigma_2(t,x)\\
&&+ \pis{C}(\eps(\dt u_1(t,x)))\cdot \eps(\dt u_1(t,x))- \pis{C}(\eps(\dt u_2(t,x)))\cdot \eps(\dt u_2(t,x)),\\
\frac{\partial\thr}{\partial\vec{n}}(t,x)|_{\partial\Omega\times(0,T)}&=&0,\\
\thr(x,0)&=&0.\end{array}\right.
\end{equation*}
Hence we can estimate
\begin{eqnarray*}
\lefteqn{\norm{\thr}_{L^p(0,T;W^{1,r}(\Omega))}+\norm{\thr}_{W{1,p}(0,T;L^r(\Omega))}}\\
%
%
&\leq&C(T)\Big(\norm{\phi(\ths_1)\dyw\dt u_1-\phi(\ths_2)\dyw\dt u_2}_{L^p(0,T;L^r(\Omega))}+ \norm{\dt\epsp_1\cdot\sigma_1-\dt\epsp_2\cdot\sigma_2}_{L^p(0,T;L^r(\Omega))}\\
&&\phantom{C(T)\Big(}+ \norm{\pis{C}(\eps(\dt u_1))\cdot \eps(\dt u_1)- \pis{C}(\eps(\dt u_2))\cdot \eps(\dt u_2)}_{L^p(0,T;L^r(\Omega))}\Big)\\
%
%
&\leq&C(T)\Big(\norm{\phi(\ths_1)\dyw\dt \ur}_{L^p(0,T;L^r(\Omega))}+ \norm{(\phi(\ths_1)-\phi(\ths_2))\dyw\dt u_2}_{L^p(0,T;L^r(\Omega))}\\
&&\phantom{C(T)\Big(}+\norm{\dt\epsp_1\cdot\Sr}_{L^p(0,T;L^r(\Omega))}+ \norm{\dt\epspr\cdot\sigma_2}_{L^p(0,T;L^r(\Omega))}\\
&&\phantom{C(T)\Big(}+ \norm{\pis{C}(\eps(\dt u_1))\cdot \eps(\dt \ur)}_{L^p(0,T;L^r(\Omega))}+ \norm{\pis{C}(\eps(\dt \ur))\cdot \eps(\dt u_2)}_{L^p(0,T;L^r(\Omega))}\Big)\\
%
%
&\leq&C(T)\Big(\norm{\phi(\ths_1)}_{L^{2p}(0,T;L^{2r}(\Omega))}\norm{\dyw\dt \ur}_{L^{2p}(0,T;L^{2r}(\Omega))}\\
&&\phantom{C(T)\Big(}+ \norm{\phi(\ths_1)-\phi(\ths_2)}_{L^{2p}(0,T;L^{2r}(\Omega))}\norm{\dyw\dt u_2}_{L^p(0,T;L^r(\Omega))}\\
&&\phantom{C(T)\Big(}+\norm{\dt\epsp_1}_{L^{2p}(0,T;L^{2r}(\Omega))}\norm{\Sr}_{L^{2p}(0,T;L^{2r}(\Omega))}+ \norm{\dt\epspr}_{L^{2p}(0,T;L^{2r}(\Omega))}\norm{\sigma_2}^2_{L^{2p}(0,T;L^{2r}(\Omega))}\\
&&\phantom{C(T)\Big(}+ \norm{\eps(\dt u_1))}_{L^{2p}(0,T;L^{2r}(\Omega))}\norm{\eps(\dt \ur)}_{L^{2p}(0,T;L^{2r}(\Omega))}\\
&&\phantom{C(T)\Big(}+ \norm{\eps(\dt \ur)}_{L^{2p}(0,T;L^{2r}(\Omega))}\norm{ \eps(\dt u_2)}_{L^{2p}(0,T;L^{2r}(\Omega))}\Big).
\end{eqnarray*}
Now we can see that for $\delta>0$ small enough it holds that $\norm{\ths_2}_{L^p(0,T;L^r(\Omega))}\leq 2\norm{\ths_1}_{L^p(0,T;L^r(\Omega))}$ if only $\norm{\thsr}_{L^p(0,T;L^r(\Omega))}\leq\delta$. Hence using estimates (\ref{eq:ap1a}) and (\ref{eq:ciaglosc1}) we can write:
\begin{eqnarray*}
\lefteqn{\norm{\thr}_{L^p(0,T;W^{1,r}(\Omega))}+\norm{\thr}_{W{1,p}(0,T;L^r(\Omega))}}\\
%
%
&\leq&C(T)\Big(\norm{\nab\dt \ur}_{L^{2p}(0,T;L^{2r}(\Omega))}+ \norm{\phi(\ths_1)-\phi(\ths_2)}_{L^{2p}(0,T;L^{2r}(\Omega))}\\
&&\phantom{C(T)\Big(}+\norm{\Sr}_{L^{2p}(0,T;L^{2r}(\Omega))}+ \norm{\dt\epspr}_{L^{2p}(0,T;L^{2r}(\Omega))}\Big)\\
%
%
&\leq&C(T)\Big(\norm{\nab\dt \ur}_{L^{s}(0,T;L^{q}(\Omega))}+ \norm{\phi(\ths_1)-\phi(\ths_2)}_{L^{s}(0,T;L^{q}(\Omega))}\\
&&\phantom{C(T)\Big(}+\norm{\Sr}_{L^{s}(0,T;L^{q}(\Omega))}+ \norm{\dt\epspr}_{L^{s}(0,T;L^{q}(\Omega))}\Big)\\
&\leq&C(T)\epsilon\leq\widetilde{\epsilon},
\end{eqnarray*}
where constant $C(T)>0$ depends on a lenghth of the time interval $(0,T)$, norms of given data and the norm $\norm{\ths_1}_{L^p(0,T;L^r(\Omega))}$. For any given $\widetilde{\epsilon}>0$ we can choose such $\delta>0$ that the estimate above holds true. Therefore operator $\pis{T}$ is continuous. 
\end{proof}
%
\begin{stw}
the operator $\pis{T}\colon L^p(0,T;L^r(\Omega))\rightarrow L^p(0,T;L^r(\Omega))$ is compact.
\end{stw}
The reason of the fact above is the Aubin-Lions lemma {\it i.e.}:
$$
L^p(0,T;W^{2,r}(\Omega))\cup W^{1,p}(0,T;L^r(\Omega))\hookrightarrow\hookrightarrow L^p(0,T;L^r(\Omega)).
$$
The embedding above together with Proposition \ref{stw:T-ciagly} give as a compactness for the operator $\pis{T}$ in required spaces.
\begin{lem}
There exists $\theta\in L^p(0,T;L^r(\Omega))$ such that $T(\theta)=\theta$.
\end{lem} 
\begin{proof}We use the Schauder fixed point theorem. Let us fixed $\ths\in L^p(0,T;L^r(\Omega))$ such that $\norm{\ths}_{L^p(0,T;L^r(\Omega))}\leq M$ for some $M>0$ which will be indicated later. We are going to prove that $\norm{\pis{T}(\ths)}_{L^p(0,T;L^r(\Omega))}=\norm{\theta}_{L^p(0,T;L^r(\Omega))}\leq M$. From estimates from Lemma \ref{lem:RHS} and Lemma \ref{lem:ap2} we immediately obtain that:
\begin{eqnarray*}
\lefteqn{\norm{\theta}_{L^p(0,T;L^r(\Omega))}}{}\\
&&\leq\norm{\theta}_{L^p(0,T;W^{2,r}(\Omega))}+\norm{\theta}_{W^{1,p}(0,T;L^r(\Omega))}\\
&\quad&\leq C(T)\left(\norm{\ths}^{2\beta}_{L^{p}(0,T;L^{r}(\Omega))}+ \norm{\ths}^{2\alpha}_{L^{p}(0,T;L^{r}(\Omega))}+ \norm{\ths}^{\beta+\alpha}_{L^{p}(0,T;L^{r}(\Omega))}\right.\\
&&\phantom{\leq C(T)M}+\left.|\Omega|^2|\Ld(0,0)|^2+\norm{\epsp_0}^2_{\Lq}+ \norm{b}^2_{W^{1,s}(0,T;\Lq)}+ \norm{g}^2_{W^{1,s}(0,T;W^{-\frac{1}{q},q}(\partial\Omega))}\right.\\
&&\phantom{\leq C(T)M}+\left.\norm{\theta_0}_{B^{2(1-1/p)}_{rp}(\Omega)}+ \norm{h}_{F^{(r-1)/(2r)}_{pr}(0,T;L^r(\partial\Omega))} + \norm{h}_{L^p(0,T;W^{1-1/r,r}(\partial\Omega))}\right)\\
&\quad&\leq D(T)\left(M^{2\beta}+ M^{2\alpha}+ M^{\beta+\alpha}+ 1\right)\\
&\quad&\leq E(T)+\frac{1}{2}M\\
&\quad&\leq M
\end{eqnarray*}
if only
$$
M\geq 2E(T).
$$
Here the constant $E(T)>0$ does depend only on given boundary and initial data, entries of the operators $\pis{D}$ and $\pis{C}$, geometry of the domain $\Omega$ and the length of the time interval $(0,T)$.

Hence for such a constant $M$ the compact operator $\pis{T}\left(\mathbf{B}(0,M)\right)\subset \mathbf{B}(0,M)$, where $$\mathbf{B}(0,M):=\{\xi\in L^p(0,T;L^r(\Omega))\,|\, \norm{\xi}_{L^p(0,T;L^r(\Omega))}\leq M\}$$. Finally we employ the Schauder fixed point theorem to end the proof.
\end{proof}
\begin{wni}
The consideration above proves the main result of this paper {\it i.e.} Theorem \ref{twr:main}.
\end{wni}
\begin{uwa}
It is possible to consider the model with non-homogeneous boundary condition in Dirichlet part. {\it i.e.}
$$
u(t,x)=f(t,x)\qquad \text{for}\; (t,x)\in(0,T)\times\Gamma_0.
$$
It is enough to assume that the function $f$ can be expanded to a function $F\in W^{1,s}(0,T;W^{1,q}(\Omega;\Rz^3))$.
\end{uwa}
\begin{uwa}
We would treat the field $\Lambda$ in evolution equation for the plastic part of the strain tensor as Yosida approximation to subdifferential, but in that case we need to obtain estimates independent of the Lipschitz constant $L_1$. It seems it should be possible if we handle with  Boccardo-Gallou\"et theory for the parabolic equation with the right-hand side in $L^1$ (c.f. \cite{bocc89}).
\end{uwa}
%
%

\bibliographystyle{elsarticle-harv}
\bibliography{szkic4}{}

\end{document}